\documentclass[12pt]{article}
\usepackage{amsfonts,amsbsy,amssymb,amsmath,amsthm}
\usepackage{a4}
\usepackage{notestyle}
\usepackage{fourier,dejavu}
\usepackage{graphicx}
\usepackage[english]{babel}
\usepackage{datetime}
\usepackage[colorinlistoftodos]{todonotes}
\usepackage[numbers]{natbib}

\begin{document}
\title{Polynomial overreproduction by Hermite subdivision operators,
  and $p$--Cauchy numbers}
\author{
{Caroline Moosm\"uller}\thanks{Department of Mathematics, University of California, San Diego, 9500 Gilman Drive, La Jolla, CA 92093, USA. \mbox{\texttt{cmoosmueller@ucsd.edu}}} 
  \and
{Tomas Sauer}\thanks{
Lehrstuhl f\"ur Mathematik mit Schwerpunkt Digitale
  Signalverarbeitung \& FORWISS, Universit\"at Passau,
  Fraunhofer IIS Research Group on Knowledge Based Image
  Processing, Innstr.~43,
  94032 Passau, Germany. \mbox{\texttt{tomas.sauer@uni-passau.de}}}}
\date{}
\maketitle

\begin{abstract}
  We study the case of Hermite subdivision operators satisfying a
  spectral condition of order greater than their size. We show that
  this can be characterized by operator factorizations involving
  Taylor operators and difference factorizations of a rank one vector
  scheme. Giving explicit expressions for the factorization operators,
  we put into evidence that the factorization only depends on the order of the
  spectral condition but not on the polynomials that define it. We further show that the
  derivation of these operators is based on an interplay between Stirling numbers and $p$--Cauchy numbers (or generalized
  Gregory coefficients).
    \par\smallskip\noindent
  {\bf Keywords:} Hermite subdivision schemes; operator factorization; $p$--Cauchy numbers
  \par\smallskip\noindent
  {\bf MSC:} 65D15;  41A58; 11B73 
\end{abstract}


\section{Introduction}
A dyadic stationary subdivision operator $S_a$ acts on a  sequence $c :
\ZZ \to \RR$ by means of the convolution like and hence stationary
operation
$$
c \mapsto S_a c := \sum_{\alpha \in \ZZ} a ( \cdot - 2\alpha ) \, c(\alpha).
$$
Here $a$, the so-called \emph{mask} of the subdivision operator, is a
finitely supported sequence. There are various ways of generalizing
subdivision operators. For example, one can consider several variables, dilation factors greater than $2$ or even expansive dilation
matrices, or vector- or matrix-valued data which requires the mask
to be a finitely supported matrix-valued sequence, cf.\ \cite{cavaretta91}. A \emph{subdivision
scheme} is an iteration of subdivision operators that may even depend
on the level of iteration, where the $n$th iteration is seen as data
defined on the grid $2^{-n} \ZZ$. Since these grids get finer and
finer, there is the concept of a \emph{limit function} of subdivision
schemes, cf. \cite{cavaretta91}.

\emph{Hermite subdivision} is a special case of subdivision operators
with matrix masks acting on
vector data, where the components of these vectors are interpreted as
consecutive derivatives. Such schemes have been considered and
analyzed first in \cite{dyn95,merrien92}.
The chain rule then enforces a subdivision
process of a mildly level-de\-pendent form that consists of a left and
right multiplication by dyadic diagonal matrices. Also the notion of
\emph{convergence} is special for Hermite subdivision schemes: If the
input data is in $\RR^{d+1}$, the limit function is vector-valued of size $d+1$ and
consists of a $C^d$ function and its derivatives up to order $d$.

It is well--known in subdivision theory \cite{cavaretta91,dyn92} that
the regularity of a limit function implies the preservation of certain
polynomials by the subdivision scheme. For Hermite
subdivision schemes this is usually formulated in terms of the \emph{spectral
  condition} and has been related to Taylor polynomials in
\cite{dubuc09}. In
\cite{merrienSauer12:_from_hermit} it is shown that the spectral condition is essentially
equivalent to an operator factorization of the form
\begin{equation}
  \label{eq:TSASBT}
  T S_{\Ab} = S_{\Bb} T  
\end{equation}
where $T$ is the so--called \emph{Taylor operator}. $T$ is a discrete version
of the Taylor formula and relates successive entries of vector-valued data
in accordance with the assumption that they are consecutive
derivatives. Moreover, the contractivity of $S_{\Bb}$ plays an important role in the analysis of convergence, cf. \cite{merrienSauer12:_from_hermit}.

In \cite{MerrienSauer18S} it is conjectured that convergence implies a generalized spectral condition of order at least $d$ to be satisfied. This is in accordance with similar results for scalar subdivision schemes, cf.\ \cite{cavaretta91}. Therefore, if one is interested in Hermite schemes of regularity $n > d$, that is, limit functions consisting of a $C^n$ function and its first $d$ derivatives, the Hermite scheme should satisfy a spectral condition of order at least $n$. Schemes of regularity $n>d$ are considered in e.g.\ \cite{conti14,jeong17,moosmueller18}.

We call this phenomenon \emph{polynomial overreproduction} and it is the main topic of this paper. We describe conditions under which the subdivision operator $S_{\Ab}$ satisfies a spectral
condition of degree \emph{higher} than $d$, providing a generalization of \cite{moosmueller18b}. It turns out that
this property fits well into the existing theory: $S_\Ab$ has
to have a factorization by means of a Taylor operator as in
\eqref{eq:TSASBT} and the \emph{rank one} vector subdivision
scheme $S_\Bb$ has 
to be factorizable in the sense defined in
\cite{MicchelliSauer97a,micchelli98}.
There is, however, a peculiarity:
The matrices that appear in the factorizations of rank one schemes
are derived from the spectral condition, but do not depend on the concrete choice of $\Ab$.

The paper is organized as follows. We start by introducing notation and give
detailed definitions of the above properties in
Section~\ref{sec:Notation}; factorizations of subdivision operators
are revised in Section~\ref{sec:factorization}. In
Section~\ref{sec:Stirling} we introduce Stirling numbers and their connection to $p$--Cauchy numbers.
Based on the technical preliminaries of Section~\ref{sec:aux},
the main result of the paper, namely the factorization with respect to
the \emph{augmented Taylor operator}, is given in
Section~\ref{sec:augmented} with a rather short proof.

\section{Notation and subdivision schemes}
\label{sec:Notation}
Throughout this paper, $d$ denotes an integer, and $d\geq 1$. Vectors
in $\RR^{d+1}$ are written as $\cb$, that is, with boldface lowercase
letters, while matrices $\Ab$ are written with boldface uppercase
letters. The standard basis in $\RR^{d+1}$ is denoted by
$\eb_0,\ldots,\eb_d$. The identity matrix of dimension $d+1 
$ is denoted by $\Ib_{d+1}$. We also use the Matlab-like notation
$\cb_{k:\ell}$ to extract subvectors. 
Furthermore, for a vector $\cb \in \RR^{d+1}$ we introduce the notation
$
    \hat{\cb} = \left(
      \cb_0,
      \ldots,
      \cb_{d-1},
      0
   \right)^T
$
for the canonical embedding of $\cb$ into $\RR^{d+1+k}$, $k \ge 1$.

The space of all polynomials in one variable is written as $\Pi$, while $\Pi_n$ denotes all such polynomials with degree at most $n$.

By $\ell^{d+1}(\ZZ)$ we denote the space of all sequences $\cb: \ZZ \to \RR^{d+1}$, while $\ell^{(d+1)\times (d+1)}(\ZZ)$ is the space of matrix-valued sequences $\Ab: \ZZ \to \RR^{(d+1)\times (d+1)}$. We use the same notation for vectors (matrices) and sequences of vectors (matrices); it will be clear from the context what is meant.
The notation $\ell^{d+1}_{00}(\ZZ)$ and $\ell^{(d+1)\times
  (d+1)}_{00}(\ZZ)$ is used to denote sequences with finite support.

To distinguish them from input data for subdivision schemes, we denote
sequences of vector valued parameters by $\cb_n, n \in \NN$, in
accordance with the notation $\eb_0,\dots,\eb_d$ of the unit coordinate vectors.
The
$k-$th entry of an element of such a sequence is accessed by
$\cb_{n,k}, k=0,\ldots,d, n\in \NN$.

The \emph{forward difference operator} $\Delta$ is used both in the context of functions and sequences. If $f$ is a function, then 
$(\Delta f)(x) = f(x+1)-f(x), x\in \RR$. For  $\cb \in \ell^{d+1}(\ZZ)$ we have $(\Delta \cb)(\alpha)=\cb(\alpha+1)-\cb(\alpha), \alpha \in \ZZ$. Higher order forward difference operators are defined by $\Delta^n=\Delta(\Delta^{n-1})$, $n\geq 1$, with $\Delta^0=\operatorname{id}$.

A \emph{stationary subdivision operator} with \emph{mask} $\Ab \in \ell^{(d+1)\times (d+1)}_{00}(\ZZ)$ is a map $S_{\Ab}:\ell^{d+1}(\ZZ) \to \ell^{d+1}(\ZZ)$ defined by
\begin{equation*}
    \left(S_{\Ab}\cb\right)(\alpha)=\sum_{\beta \in \ZZ}\Ab(\alpha-2\beta)\cb(\beta), \quad \alpha \in \ZZ, \quad \cb \in \ell^{d+1}(\ZZ).
\end{equation*}
We consider a vector $\cb \in \RR^{d+1}$ as a constant sequence, so that $S_{\Ab}\cb$ means the application of $S_{\Ab}$ to the constant sequence $\cb(\alpha) = \cb, \alpha \in \ZZ$.

A level-dependent \emph{subdivision scheme} $(S_{\Ab^{[n]}},n\in \NN)$ is the procedure of iteratively constructing vector-valued sequences by
\begin{equation}\label{eq:sds}
\cb^{[n+1]}=S_{\Ab^{[n]}}\cb^{[n]},\quad n\in \NN,
\end{equation}
from initial data $\cb^{[0]} \in \ell^{d+1}(\ZZ)$.
In this paper we consider two cases of such subdivision schemes based
on stationary subdivision operators:
\emph{vector subdivision schemes} which use the same mask in every
iteration level, i.e.\ $\Ab^{[n]}=\Ab, n \in \NN$,
cf. \cite{micchelli98}, and
\emph{Hermite subdivision schemes} which use the mildly level-dependent masks
\begin{equation}\label{eq:stationary_mask}
\Ab^{[n]}=\Db^{-n-1}\,\Ab\,\Db^{n}
\end{equation}
where $\Db=\operatorname{diag}\left(1, 2^{-1},\ldots,2^{-d} \right)$ and $\Ab \in \ell^{(d+1) \times (d+1)}_{00} (\ZZ)$ is fixed. In Hermite subdivision, the data $c^{[n]}$ represents function and consecutive derivative values at $2^{-n}\alpha,\alpha\in \ZZ$, leading to the mask \eqref{eq:stationary_mask} via the chain rule.

For $p \in \Pi$ we define the vector-valued function
\begin{equation}\label{eq:spectral_polynomial}
\vb(p)(x) := \left( p^{(k)}(x) : k = 0,\dots,d \right)^T, \quad x\in \RR.
\end{equation}
We also consider $\vb(p)$ as a sequence in $\ell^{d+1}(\ZZ)$, by evaluating at integers only. The particular meaning of $\vb(p)$ will be clear from the context.

A Hermite subdivision scheme is said to satisfy the \emph{spectral condition of order $n\geq d$} if there exist $p_k \in \Pi_k$, normalized as $p_k (x) =
\frac1{k!} x^k + \cdots$, such that
\begin{equation}
  \label{eq:SpecCond}
  S_\Ab \vb (p_k) = 2^{-k} \vb (p_k), \qquad k=0,\dots,n.
\end{equation}
The spectral condition for $n=d$ has first been introduced by
\cite{dubuc09}, see also \cite{merrienSauer12:_from_hermit}. The case
$n > d$ is a higher order spectral condition studied in
\cite{conti14}, and we denote it by \emph{polynomial overreproduction}. The recent paper \cite{MerrienSauer18S} introduces
\emph{spectral chains}, which generalize \eqref{eq:SpecCond}. We
briefly discuss spectral chains in Section \ref{sec:augmented}.

While the spectral condition of order $d$ is important for factorization of Hermite subdivision operators \cite{merrienSauer12:_from_hermit}, it has been shown that it is not necessary for convergence \cite{merrienSauer17:_exten_hermit,MerrienSauer18S}.

\section{Factorization of subdivision operators}\label{sec:factorization}
The factorization of subdivision operators is a standard method for
proving convergence of the associated subdivision schemes and
regularity of their limits.
More precisely, convergence of subdivision schemes can often be
characterized by a factorization and \emph{contractivity} of the
factor scheme while regularity of the limit functions is described by
a factorization and \emph{convergence} of the factor scheme. It is
important, however, to emphasize that the nature of the factorization
has to be adapted to the nature of the subdivision scheme. In
particular, although vector and Hermite subdivision schemes both act
by means of matrix masks on vector valued data, the associated
factorizations are of a siginificantly different nature that reflects
the different conceptual nature of the schemes.

In this paper, we are concerned with
factorizations of rank $1$ vector schemes  as derived in 
\cite{charina05,MicchelliSauer97a,micchelli98,sauer02} and Taylor
factorizations of Hermite schemes
\cite{conti16,merrienSauer12:_from_hermit,MerrienSauer18S}. We now
introduce these concepts.

Following \cite{micchelli98}, for a subdivision operator $S_{\Bb}$, we define
\begin{equation}
  \cE_{\Bb} = \left\{\cb \in \RR^{d+1}: S_{\Bb}\cb = \cb \right\},
\end{equation}
which is the eigenspace (of constant sequences) of $S_{\Bb}$ with
respect to the eigenvalue $1$. The dimension $\dim \cE_\Bb$ is called
the \emph{rank} of the subdivision scheme. In this paper we are only
concerned with \emph{rank $1$ schemes}, i.e.\ operators $S_{\Bb}$
satisfying $\dim\, \cE_{\Bb}=1$, cf.\ \cite{MicchelliSauer97a}.
We call a matrix $\Vb = \left(v_0, \ldots, v_d\right)$ with $v_j \in \RR^{d+1}, j=0,\ldots,d$, an $\cE_{\Bb}$-\emph{generator} if $\left\{v_0, \ldots, v_d\right\}$ is a basis of $\RR^{d+1}$ and if $v_d$ spans $\cE_{\Bb}$.

With the operator
\begin{equation}
D :=
\begin{pmatrix}
  \Ib_d & \\ & \Delta
\end{pmatrix}
\end{equation}
the following result has been shown, cf. \cite{MicchelliSauer97a,micchelli98}:
\begin{lemma}\label{lem:rank1_factorization}
Let $S_{\Bb}$ be a subdivision operator with $\dim\,\cE_{\Bb}=1$. If $\Vb$ is an $\cE_{\Bb}$-generator, then there exists a subdivision operator $S_{\Cb}$ such that
\begin{equation*}
    D\Vb^{-1}\,S_{\Bb}=S_{\Cb}\,D\Vb^{-1}.
\end{equation*}
Furthermore, $\dim\,\cE_{\Cb}=1$.
\end{lemma}
From \cite{merrienSauer12:_from_hermit} recall the
\emph{(incomplete) Taylor operator} 
\begin{equation*}
T_d =
\begin{pmatrix}
  \Delta & -1 & -\frac12 & \dots & -\frac{1}{d!} \\
  & \ddots & \ddots & \ddots & \vdots \\
  & & \Delta & -1 & -\frac{1}{2!} \\
  & & & \Delta & -1 \\
  & & & & 1 \\
\end{pmatrix}
\end{equation*}
and the \emph{complete Taylor operator}
\begin{equation*}
    \widetilde T_d =
D \, T_d = \begin{pmatrix}
  \Delta & -1 & -\frac12 & \dots & -\frac{1}{d!} \\
  & \ddots & \ddots & \ddots & \vdots \\
  & & \Delta & -1 & -\frac{1}{2!} \\
  & & & \Delta & -1 \\
  & & & & \Delta \\
\end{pmatrix}.
\end{equation*}
We also consider the following operator which has been defined and studied in \cite{dubuc09}:
$$
 T_d' = \begin{pmatrix}
  \Delta & -1 & \dots & -\frac{1}{(d-1)!}&0 \\
  & \ddots  & \ddots & \vdots & \vdots\\
  & &  \Delta & -1 & 0 \\
  & &  & \Delta & 0 \\
  & &  & & 1 \\
\end{pmatrix}.
$$
We furthermore define $\widetilde{T}_0=\Delta$ and $T_0=T_0'=\operatorname{id}$.
Generalizations of these Taylor operators have been introduced in
\cite{MerrienSauer18S}; we discuss them in Section \ref{sec:augmented}.


It has been shown in \cite[Theorem 4]{merrienSauer12:_from_hermit} that a subdivision operator $S_{\Ab}$ satisfying the
spectral condition of order $d$ \eqref{eq:SpecCond} can be factorized with respect to the Taylor operator: There exists a subdivision operator $S_{\Bb}$ such that
\begin{equation}\label{eq:Taylor_factorization}
T_d S_\Ab = 2^{-d} S_\Bb T_d.
\end{equation}
If $S_{\Ab}$ factorizes as in \eqref{eq:Taylor_factorization}, but stepwise, i.e.\  with respect to operators 
\begin{equation*}
\begin{pmatrix}
T_k & \\
& \Ib_{d-k}
\end{pmatrix}, \quad k = 0,\ldots, d,
\end{equation*}
then this
is even a characterization of the spectral condition of order $d$ \eqref{eq:SpecCond},
cf. \cite[Corollary 2.12]{merrienSauer17:_exten_hermit}.
Furthermore, $\cE_{\Bb}$ is spanned by $\eb_d$. Therefore $\Vb = \Ib_{d+1}$ is an $\cE_{\Bb}$-generator and by Lemma \ref{lem:rank1_factorization} there exists a subdivision operator $S_{\Cb}$ such that
\begin{equation*}
    D\, S_{\Bb} = S_{\Cb}\,D.
\end{equation*}
The latter implies
\begin{equation*}
  \widetilde T_d S_\Ab
   =  D \, T_d S_\Ab = 2^{-d} D S_\Bb T_d = 2^{-d} S_{\Cb} D T_d 
   =  2^{-d} S_{\Cb} \widetilde T_d,  
\end{equation*}
which is the complete Taylor factorization of \cite[Theorem 4]{merrienSauer12:_from_hermit}:
\begin{equation}\label{eq:level0}
  \widetilde T_d S_\Ab = 2^{-d} \, S_{\Cb} \widetilde T_d.
\end{equation}
In this paper we prove a generalization of \eqref{eq:level0} to operators $S_{\Ab}$ which satisfy the spectral condition \eqref{eq:SpecCond} for $n>d$ (Theorem \ref{thm:factorization}). In particular we prove that every such operator factorizes with respect to the \emph{augmented Taylor operator} of order $n$:
\begin{definition}[Augmented Taylor operators] \label{def:augmented}
For $d\geq 1$ and $n\geq d$ we define the \emph{augmented Taylor operator} of order $n$ by
\begin{equation*}
\widetilde T_d^n :=
  \begin{pmatrix}
    \widetilde{T}_{d-1} & -\displaystyle{\sum_{k=0}^{n-d}} G_k^{d:1} \Delta^{k} \\
    & \Delta^{n+1-d} \\
  \end{pmatrix}
  =
 \begin{pmatrix}
  \Delta & -1 & -\frac12 & \dots & -\frac{1}{(d-1)!} &
  -\displaystyle{\sum_{k=0}^{n-d}} G_k^d\Delta^k\\
  & \ddots & \ddots &  & \vdots & \vdots\\
  & & \ddots & \ddots & \vdots  & \vdots\\
  & & & \Delta & -1 & -\displaystyle{\sum_{k=0}^{n-d}} G_k^2\Delta^k\\[0.2cm]
  & & & & \Delta & -\displaystyle{\sum_{k=0}^{n-d}} G_k^1\Delta^k\\[0.2cm]
  & & & &  & \Delta^{n+1-d}
\end{pmatrix},
\end{equation*}
where $G_k^{d:1}=\left(G_k^d,G_k^{d-1},\ldots,G_k^1 \right)^T$, and
$G_k^{\ell}, k \geq 0, \ell \geq 1$ are the \emph{coefficients for
  repeated integration with forward differences} \cite{salzer47}. 
\end{definition}

\begin{remark}
  Normalizing the coefficients $G_k^n$ as in
  \eqref{eq:p_Cauchy} leads to the \emph{$p$--Cauchy numbers of the
    first kind} , see \cite{Rahmani16}. Since $G_k^1$ 
  are known, among others, as \emph{Gregory coefficients},
  cf. \cite{blagouchine16}, one could call these numbers
  \emph{generalized Gregory coefficients}. We discuss them in more
  detail in Section \ref{sec:Stirling}.   
\end{remark}

The existence of such a factorization follows from combining
the Taylor factorization \eqref{eq:Taylor_factorization} of \cite{merrienSauer12:_from_hermit} with iterated factorizations for rank $1$ schemes (Lemma \ref{lem:rank1_factorization}) of
\cite{MicchelliSauer97a,micchelli98}. The contribution of this paper is the explicit computation of the augmented Taylor operators via computing $\cE_{\Bb_j}$ for every iteration $j=d,\ldots,n$ of rank $1$ factorizations. In particular, we show that the spectral condition \eqref{eq:SpecCond}, but \emph{not} the choice of spectral polynomials, already determines all $\cE_{\Bb_j}, j=d,\ldots,n$. We thus also extend the results of \cite{moosmueller18b}.

\section{Stirling and $p$--Cauchy numbers}
\label{sec:Stirling}
Following \cite{graham94}, we recall the definition of Stirling numbers.

The \emph{Stirling numbers of the first kind}, denoted by $\stiri{n}{m}$, count the numbers of ways to arrange $n$ elements into $m$ cycles. From the initial conditions
\begin{equation*}
\stiri{0}{0}=1, \quad \stiri{n}{0}=\stiri{0}{n}=0, \quad n\geq 1,
\end{equation*}
they can be computed via the following recurrence relation:
\begin{equation*}
\stiri{n+1}{m}=n\stiri{n}{m}+\stiri{n}{m-1}, \quad m\geq 1.
\end{equation*}
The \emph{signed Stirling numbers of the first kind} are defined by
\begin{equation}\label{def:signed_stir1}
    s(n,m) = (-1)^{n-m}\stiri{n}{m}.
\end{equation}
They satisfy the recurrence relation
\begin{equation}\label{signed_stir1_recur}
    s(n+1,m)=s(n,m-1)-n\, s(n,m),
\end{equation}
with initial conditions
\begin{equation*}
    s(n,n)=1, \quad s(n,m)=0 \quad \text{ if } m<n \text{ or } n<m.
\end{equation*}
The \emph{Stirling numbers of the second kind}, denoted by $\stir{n}{m}$, count the number of ways to split a set of $n$ elements into $m$ non-empty subsets.
They satisfy the following recurrence relation
\begin{equation}\label{stir2_rec}
    \stir{n+1}{m} = m \, \stir{n}{m}+\stir{n}{m-1}, \quad m \geq 1.
\end{equation}
with initial conditions
\begin{equation*}
    \stir{0}{0}=1, \quad \stir{n}{0}=\stir{0}{n}=0, \quad n \geq 1.
\end{equation*}
The Stirling numbers of the second kind can be computed using Binomial coefficients
\begin{equation*}
\stir{n}{m}=\frac{1}{m!}\sum_{j=0}^m {m \choose j} (-1)^{m-j}j^n.
\end{equation*}
We also need the following relation between the Stirling numbers of the second kind and the Binomial coefficients (see \cite[Eq. 6.15]{graham94}):
\begin{equation}\label{eq:Stir_Bino}
    \stir{n+1}{m+1}=\sum_{k=m}^n {n \choose k}\stir{k}{m}.
\end{equation}
Following \cite{salzer47}, we define the \emph{coefficients for repeated integration with forward differences}, $G_n^k$ for $k,n\geq 1$, by
\begin{equation}\label{eq:coeff_int1}
    G_n^{1}=\frac{1}{n!}\,\int_{0}^1 x(x-1)\cdots (x-n+1)
    dx, \quad n \geq 1,
\end{equation}
and
\begin{equation}\label{eq:coeff_int}
    G_n^{k}=\frac{1}{n!}\,\int_{0}^1\int_0^{x_2} \cdots \int_0^{x_k} x(x-1)\cdots (x-n+1)\,
    dx dx_k \cdots dx_2, \quad n \geq 1, k\geq 2.
\end{equation}
We also define
\begin{equation}\label{eq:coeff_0}
    G^k_0=\frac{1}{k!}, \quad k\geq 1.
\end{equation}
The coefficients $G_n^k$ are connected to the \emph{$p$--Cauchy numbers of the first kind}, $\cC_{n,p}$, defined in \cite{Rahmani16}, via
\begin{equation}\label{eq:p_Cauchy}
   \cC_{n,p-1} = n!\,p! \,G_n^p.
\end{equation}
The sequence $G_n^1$ are the \emph{Gregory coefficients}, since \eqref{eq:coeff_int1} is their well-known integral representation, see e.g.\ \cite{merlini06}. The Gregory coefficients are a well-studied sequence in number theory and are also known as the \emph{Cauchy numbers of the first kind}, the \emph{Bernoulli numbers of the second kind} and the \emph{reciprocal logarithmic numbers}, see e.g.\ \cite{blagouchine17,kowalenko10,merlini06}. 
In this sense, the coefficients in \eqref{eq:coeff_int} are a generalization of the Gregory coefficients. Another generalization of the Gregory coefficients can be found in \cite[Eq.\ (63)]{blagouchine18}.

In \cite{salzer47}, the following recursion is shown to hold:
\begin{equation}\label{eq:coeff_rec}
    G_n^k=\frac{1}{1-k}\left((n-1)G_n^{k-1}+(n+1)G_{n+1}^{k-1}\right), \quad k\geq 2, n\geq 1,
\end{equation}
compare also to the equivalent recursion for $p$--Cauchy numbers in \cite[Theorem 2.5]{Rahmani16}.
Via \eqref{eq:p_Cauchy}, Corollary 2.3 \& Theorem 2.2 of \cite{Rahmani16} imply
\begin{equation}\label{lem:stir2_coeff}
    \sum_{r=1}^{j}\stir{j}{r}\,r!\,G_r^k=\frac{1}{(j+1)\cdots(j+k)}, \quad j,k\geq 1
\end{equation}
and
\begin{equation}\label{eq:stir1_coeff}
    G_n^k=\frac{1}{n!}\sum_{j=1}^n \frac{s(n,j)}{(j+1)\cdots (j+k)}, \quad j,k\geq 1.
\end{equation}
For $k=1$, \eqref{lem:stir2_coeff} and \eqref{eq:stir1_coeff} are proved in \cite{merlini06}.
\begin{remark}
  The case $k=2$ of \eqref{eq:stir1_coeff} can also be found on oeis.org (sequence A002687 resp.\ A002688) under ``formula''.
\end{remark}

\section{Auxiliary results}\label{sec:aux}
We start by proving that the Stirling numbers of the second kind relate forward differences to derivatives:
\begin{lemma}\label{lem:high_forward_diff}
For $p \in \Pi_n, \ell \leq n,\, 1\leq k\leq n-\ell$ we have
\begin{equation*}
\frac{1}{k!}\,\Delta^{k}p^{(\ell)}=\sum_{m=k}^{n-\ell} \frac{1}{m!}\stir{m}{k}p^{(m+\ell)}.
\end{equation*}
\end{lemma}
\begin{pf}
We prove this by induction on $k$. For $k=1$
the Taylor formula gives
\begin{equation*}
\Delta p (x) = p(x+1) - p(x) = \sum_{m=1}^{n} \frac1{m!} p^{(m)}
(x)
\end{equation*}
and for $\ell \leq n$
\begin{equation}\label{eq:Delta}
\Delta p^{(\ell)} (x) = \sum_{m=1}^{n-\ell} \frac{1}{m!} p^{(\ell+m)} (x).
\end{equation}
We assume the statement is true for $k$ and prove it for $k+1$ using \eqref{stir2_rec}, \eqref{eq:Stir_Bino} and \eqref{eq:Delta}:
\begin{align*}
\Delta^{k+1}p^{(\ell)}&=\Delta \, \Delta^k p^{(\ell)}
= \Delta \sum_{m=k}^{n-\ell} \frac{k!}{m!}\stir{m}{k}p^{(m+\ell)}
= \sum_{m=k}^{n-\ell} \frac{k!}{m!}\stir{m}{k}\sum_{s=1}^{n-m-\ell}
\frac{1}{s!}p^{(s+m+\ell)}\\
&=\sum_{m=k}^{n-\ell}\sum_{s=m+1}^{n-\ell} \frac{k!}{s!}{s \choose m}\stir{m}{k} p^{(s+\ell)}
=\sum_{s=k+1}^{n-\ell}\sum_{m=k}^{s-1} \frac{k!}{s!}{s \choose m}\stir{m}{k} p^{(s+\ell)}\\
&=\sum_{s=k+1}^{n-\ell}\frac{(k+1)!}{s!}\stir{s}{k+1}p^{(s+\ell)}.
\end{align*}
This concludes the induction.
\end{pf}

\begin{definition}\label{def:yn}
Define the following vector-valued sequences for $j\geq 0$:
\begin{align*}
\ab_{j}&:= \left(\frac1{(j+d)!}, \frac{1}{(j+d-1)!},\ldots,\frac{1}{(j+1)!}, \frac{1}{j!}\right)^T ,\\
\yb_j&:= \left( G^d_j,\ldots, G^1_j, 0 \right)^T.
\end{align*}
\end{definition}
The following lemma is essential for the main result of this paper,
Theorem \ref{thm:factorization}, since it identifies the sequence
$\yb_j$ as the correct coefficients for factorization.

\begin{lemma}\label{lem:y2}
The sequences $(\yb_j, j\geq 0)$, and $(\ab_j,j\geq 0)$, from Definition \ref{def:yn} satisfy the following property
\begin{align}\label{al:y0}
    \yb_0 &= \hat{\ab}_0\\ \label{al:yj}
\sum_{m=1}^{j}\gamma^{j}_{m}\, \yb_{m} &=\hat{\ab}_{j}, \quad
j\geq 1.
  \end{align}
  where
  \begin{equation*}
     \gamma^j_m:=\frac{m!}{j!}\stir{j}{m}.
  \end{equation*}
\end{lemma}
\begin{proof}
Equation \eqref{al:y0} follows from the definition of $G_0^k,k=1,\ldots,d$, 
in \eqref{eq:coeff_0}.

For $j\geq 1$ and $k=1,\ldots,d$ equation \eqref{al:yj} is equivalent to
\begin{align*}
    \sum_{m=1}^{j}\gamma^{j}_{m}\, \yb_{m,\,k}=\frac{1}{(j+k)!}
\Longleftrightarrow  &\,  \sum_{m=1}^{j}\stir{j}{m}\,m!\,
\yb_{m,\,k}=\frac{j!}{(j+k)!}=\frac{1}{(j+1)\cdots (j+k)}
\end{align*}
Since $\yb_{m,k}=G_m^k$ for $k=1,\ldots,d$, \eqref{al:yj} is true by \eqref{lem:stir2_coeff}.
For $k=0$, \eqref{al:yj} is correct because both sides equal $0$.
\end{proof}

\begin{remark}\label{rem:Taylor}
Lemma \ref{lem:y2} implies $\widetilde T_d^d= \widetilde T_d$.
\end{remark}

\begin{lemma}\label{lem:augmented_iterated}
For $d\geq 1$ and $j\geq d,$ the augmented Taylor operator satisfies
  $$
  \widetilde{T}_d^j = D \left( \Ib_d - \yb_{j-d} \eb_d^T \right) \cdots D\left( \Ib_d -
    \yb_{0} \eb_d^T \right) T_d',
  $$
with $(\yb_j,j\geq 0)$ from Definition \ref{def:yn}.
\end{lemma}
\begin{pf}
Recall from Definition \ref{def:yn} that 
\begin{equation*}
    \yb_j= \left( G^d_j,\ldots, G^1_j, 0 \right)^T
    = \left( G^{d:1}_j,0 \right)^T
\end{equation*}
and from Lemma \ref{lem:y2} that $\yb_0=\hat{\ab}_0$.
Furthermore, note that for any vector $\cb \in \RR^{d+1}$ with $\cb_d=0$ we have
\begin{equation*}
    D \left( \Ib_d - \cb \eb_d^T \right)=
    \begin{pmatrix}
      \Ib_{d-1} & \\
      & \Delta
    \end{pmatrix}
    \begin{pmatrix}
      \Ib_{d-1} & -\cb_{0:d-1}\\
      & 1
    \end{pmatrix}
=
\begin{pmatrix}
      \Ib_{d-1} & -\cb_{0:d-1}\\
      & \Delta
\end{pmatrix}.
\end{equation*}

We prove the Lemma by induction on $j$. For $j=d$, by Remark \ref{rem:Taylor} we have
\begin{equation*}
    \widetilde{T}^d_d=\widetilde{T}_d=
    \begin{pmatrix}
      \widetilde{T}_{d-1} & -\ab_{0,0:d-1} \\
      & \Delta
    \end{pmatrix}
    =
        \begin{pmatrix}
      \Ib_{d-1} & -\yb_{0,0:d-1} \\
      & \Delta
    \end{pmatrix}
        \begin{pmatrix}
      \widetilde{T}_{d-1} &  \\
      & 1
    \end{pmatrix}
    = 
D \left( \Ib_d - \yb_0 \eb_d^T \right)
T_d'.
\end{equation*}
Assume that the Lemma is true for $j$, we prove it for $j+1$.
\begin{align*}
    \widetilde{T}_d^{j+1}&=
      \begin{pmatrix}
    \widetilde{T}_{d-1} & -\displaystyle{\sum_{k=0}^{j+1-d}} G_k^{d:1} \Delta^{k} \\
    & \Delta^{j+2-d} \\
  \end{pmatrix}
  =
        \begin{pmatrix}
    \Ib_{d-1} & -\yb_{j+1-d,0:d-1}\\
    & \Delta \\
  \end{pmatrix}
        \begin{pmatrix}
    \widetilde{T}_{d-1} & -\displaystyle{\sum_{k=0}^{j-d}} G_k^{d:1} \Delta^{k} \\
    & \Delta^{j+1-d} 
  \end{pmatrix}\\
  & =
  D \left( \Ib_d - \yb_{j+1-d} \eb_d^T \right)
  D \left( \Ib_d - \yb_{j-d} \eb_d^T \right) \cdots D\left( \Ib_d -
    \yb_{0} \eb_d^T \right) T_d',
\end{align*}
which concludes the induction step.
\end{pf}
The next lemma follows from \cite{merrienSauer12:_from_hermit} and Lemma \ref{lem:high_forward_diff}:
\begin{lemma}\label{lem:Taylor_applied}
 For $p\in \Pi$ with $\operatorname{deg}(p)=n >d $ we have
 \begin{equation*}
  \widetilde{T}_d\vb(p)= \sum_{k=1}^{n-d}\ab_{k} p^{(k+d)}.
 \end{equation*}
 If $n\leq d$ then $ \widetilde{T}_d\vb(p)=0$.
\end{lemma}
 We write the polynomial of Lemma \ref{lem:Taylor_applied} in the following
form
\begin{equation*}
  \label{eq:PolyRewrite}
  \sum_{k=1}^{n-d} \ab_k \, p^{(k+d)} = \eb_d q + \sum_{k=1}^{n-d} \hat \ab_k \,
  p^{(k+d)},  
\end{equation*}
where
\begin{equation}\label{eq:q}
 q = \sum_{k=1}^{n-d} \ab_{k,d}p^{(k+d)}.
\end{equation}
If $\operatorname{deg}(p)=n>d$ then $\operatorname{deg}(q)=n-d-1$.
\begin{lemma}\label{lem:q}
 For $n>d$, $0\leq k < n-d$ and the polynomial $q$ from \eqref{eq:q} we have:
 \begin{equation*}
   \Delta^k q = \sum_{s=k+1}^{n-d} \gamma_{k+1}^s p^{(s+d)},
 \end{equation*}
 with $\gamma$ defined in Lemma \ref{lem:y2}.
\end{lemma}
\begin{proof}
Note that the result is true for $k=0$. For $k\geq 1$ we use
Definition \ref{def:yn}, Lemma \ref{lem:high_forward_diff},
\eqref{stir2_rec}, and \eqref{eq:Stir_Bino}:
\begin{align*}
  \lefteqn{\frac{1}{k!}\Delta^k q 
                                    = \frac{1}{k!}
  \sum_{\ell=1}^{n-d-k} \ab_{\ell,d} \Delta^k p^{(\ell + d)} 
   = \sum_{\ell=1}^{n-d-k} \ab_{\ell,d}
  \sum_{m=k}^{n-d-\ell} \frac{1}{m!}\stir{m}{k}p^{(m+\ell+d)}}\\
  & = \sum_{\ell=1}^{n-d-k}\sum_{m=k}^{n-d-\ell}
   \frac{1}{\ell! \, m!}\stir{m}{k}p^{(m+\ell+d)}
   = \sum_{\ell=1}^{n-d-k}\sum_{s=k+\ell}^{n-d}
   \frac{1}{\ell! \, (s-\ell)!}\stir{s-\ell}{k}p^{(s+d)}\\
  & = \sum_{s=k+1}^{n-d}\sum_{\ell=1}^{s-k}
  \frac{1}{\ell! \, (s-\ell)!}\stir{s-\ell}{k}p^{(s+d)}
   = \sum_{r=1}^{n-d-k}\sum_{\ell=1}^{r}
  \frac{1}{\ell! \, (r+k-\ell)!}\stir{r+k-\ell}{k}p^{(r+k+d)}\\
  & = \sum_{r=1}^{n-d-k}\sum_{s=k}^{r+k-1}
  \frac{1}{ (r+k-s)!s!}\stir{s}{k}p^{(r+k+d)}
   = \sum_{r=1}^{n-d-k}\frac{1}{(r+k)!} \sum_{s=k}^{r+k-1}
  {r+k \choose s}\stir{s}{k}p^{(r+k+d)}\\
  & = \sum_{r=1}^{n-d-k} \frac{1}{(r+k)!}
  \left(\stir{r+k+1}{k+1}-\stir{r+k}{k} \right) p^{(r+k+d)}
   = \sum_{r=1}^{n-d-k} \frac{(k+1)}{(r+k)!}\stir{r+k}{k+1}p^{(r+k+d)}.
  \end{align*}
This implies
\begin{equation*}
 \Delta^k q  = \sum_{r=1}^{n-d-k} \gamma_{k+1}^{r+k}p^{(r+k+d)}
  = \sum_{s=k+1}^{n-d} \gamma_{k+1}^{s}p^{(s+d)}. \qedhere
\end{equation*}
\end{proof}

\begin{lemma}\label{lem:IterLemma2}
  For $p \in \Pi, \operatorname{deg}(p)=n$, $n>d$ and $(\cb_k,k\geq 1)$
  such that $\cb_{k,d}=0$ for all $k$, we have
  \begin{equation*}
    D \left( \Ib_d - \cb_j \eb^T_d \right) \cdots D \left( \Ib_d - \cb_{1}
      \eb^T_d \right) \sum_{k=1}^{n-d} \ab_k \, p^{(k+d)} 
    =\, \eb_d \Delta^{j} q - \sum_{k=0}^{j-1}\hat{\cb}_{k+1}\Delta^{k}q
    +\sum_{k=1}^{n-d} \hat \ab_k \, p^{(k+d)},
  \end{equation*}    
 for some $1\leq j \leq n-d$.
\end{lemma}
\begin{proof}
We prove this lemma by induction on $j$. 
First note that the operator $D \left( \Ib - \cb \eb^T_d \right)$ for 
any $\cb$ with $\cb_d=0$,
acts as the identity operator on vectors with last component equal to $0$.
Therefore
\begin{align*}
 D \left( \Ib - \cb \eb^T_d \right)\sum_{k=1}^{n-d}\ab_{k} p^{(k+d)}&=
 D \left( \Ib - \cb \eb^T_d \right)\eb_d q + \sum_{k=1}^{n-d} \hat \ab_k
 = 
 \begin{pmatrix}
      \Ib_{d-1} & -\cb_{0:d-1}\\
      & \Delta
\end{pmatrix}
\begin{pmatrix}
 0 \\
 q
\end{pmatrix}
+ \sum_{k=1}^{n-d} \hat \ab_k\\
& = 
\begin{pmatrix}
 -\cb_{0:d-1}q \\
 \Delta q
\end{pmatrix}
+ \sum_{k=1}^{n-d} \hat \ab_k\\
&=
\eb_d\Delta q -\hat{\cb}q + \sum_{k=1}^{n-d} \hat \ab_k.
\end{align*}
This proves the case $j=1$. 
Assume that the lemma is true for $j$, we prove it for $j+1$.
\begin{align*}
&D \left( \Ib_d - \cb_{j+1} \eb^T_d \right) \cdots D \left( \Ib_d - \cb_{1}
      \eb^T_d \right) \sum_{k=1}^{n-d} \ab_k \, p^{(k+d)} \\
    &=D \left( \Ib_d - \cb_{j+1} \eb^T_d \right)\left( \eb_d \Delta^{j} q 
    - \sum_{k=0}^{j-1}\hat{\cb}_{k+1}\Delta^{k}q
    +\sum_{k=1}^{n-d} \hat \ab_k \, p^{(k+d)}\right)\\
    &=
 \begin{pmatrix}
      \Ib_{d-1} & -\cb_{j+1,0:d-1}\\
      & \Delta
\end{pmatrix} 
\begin{pmatrix}
 0 \\
 \Delta^j q
\end{pmatrix}
- \sum_{k=0}^{j-1}\hat{\cb}_{k+1}\Delta^{k}q
 +\sum_{k=1}^{n-d} \hat \ab_k \, p^{(k+d)}\\
 &=
 \begin{pmatrix}
  -\cb_{j+1,0:d-1}\Delta^{j}q\\
  \Delta^{j+1}q
 \end{pmatrix}
 - \sum_{k=0}^{j-1}\hat{\cb}_{k+1}\Delta^{k}q
     +\sum_{k=1}^{n-d} \hat \ab_k \, p^{(k+d)}\\
&= \eb_d  \Delta^{j+1}q  - \sum_{k=0}^{j}\hat{\cb}_{k+1}\Delta^{k}q
     +\sum_{k=1}^{n-d} \hat \ab_k \, p^{(k+d)},    
\end{align*}
which concludes the induction step.
\end{proof}

\noindent 
Lemma \ref{lem:q} also has the following consequence.

\begin{corollary}\label{cor:IterCor}
With notation as in Lemma \ref{lem:IterLemma2} we have
   \begin{equation*}
    D \left( \Ib_d - \cb_j \eb^T_d \right) \cdots D \left( \Ib_d - \cb_{1}
      \eb^T_d \right) \sum_{k=1}^{n-d} \ab_k \, p^{(k+d)} 
    =\, \eb_d \Delta^{j} q 
    +\sum_{s=1}^{n-d} \left( \hat \ab_s 
    -\sum_{k=1}^{\min\{ s,j\}}\gamma_{k}^s \hat{\cb}_k \right) p^{(s+d)}.
  \end{equation*} 
\end{corollary}

\begin{lemma}\label{Lem:IterConst}
For $p \in \Pi, \operatorname{deg}(p)=n$, $n>d$, normalized such that
$p(x)= \frac{1}{n!}x^n+\ldots,$
and $(\yb_k,k\geq 1)$ from Definition \ref{def:yn}, we have
  \begin{equation*}
    D \left( \Ib_d - \yb_{n-d-1} \eb^T_d \right) \cdots D \left( \Ib_d - \yb_{1}
      \eb^T_d \right) \sum_{k=1}^{n-d} \ab_k \, p^{(k+d)} 
    = \eb_d + \yb_{n-d}.
  \end{equation*}     
\end{lemma}

\begin{pf}
Lemma \ref{lem:q} implies $\Delta^{n-d-1}q=p^{(n)}=1$, since $p$ is normalized.
Corollary \ref{cor:IterCor} and Lemma \ref{lem:y2} now imply
\begin{align*}
&D \left( \Ib - \yb_{n-d-1} \eb^T_d \right) \cdots D \left( \Ib - \yb_{1}
      \eb^T_d \right) \sum_{k=1}^{n-d} \ab_k \, p^{(k+d)}\\
&= \eb_d \Delta^{n-d-1} q
    +\sum_{s=1}^{n-d} \left(\hat \ab_s -
    \sum_{k=1}^{\min\{s,n-d-1\}} \gamma^{s}_{k}\, \yb_{k} \right) p^{(s+d)}\\
&=   \eb_d 
    +\sum_{s=1}^{n-d-1} \left(\hat \ab_s -\sum_{k=1}^{s}
    \gamma^{s}_{k}\, \yb_{k} \right) p^{(s+d)}
    + \left(\hat \ab_{n-d} -\sum_{k=1}^{n-d-1}
    \gamma^{n-d}_{k}\, \yb_{k} \right) p^{(n)}\\
&= \eb_d + \yb_{n-d}.
\end{align*}
This concludes the proof.
\end{pf}

\noindent
Finally,
Lemma \ref{lem:augmented_iterated} and Lemma \ref{Lem:IterConst} imply
the following result.

\begin{corollary}\label{cor:augmented_eigenvectors}
With notation as in Lemma \ref{Lem:IterConst} we have
\begin{equation*}
 \widetilde{T}_d^{n-1}\vb(p)=\eb_d+\yb_{n-d}.
\end{equation*}
\end{corollary}

\section{Factorization with respect to the augmented Taylor operator}
\label{sec:augmented}

We can now apply the results from the preceding sections to describe
the factorization for Hermite schemes with a spectral condition of possibly
higher order $n \ge d$. It is based on the augmented Taylor operator,
hence on combining Taylor operators with appropriate difference
operators of rank $1$.

\begin{theorem}[Main result]\label{thm:factorization}
 If $S_\Ab$ satisfies the spectral condition
  \eqref{eq:SpecCond} with $n\geq d$,
  then there exist subdivision operators $S_{\Bb_j}, j=d,\ldots,n$, such that we can factorize
  \begin{equation}
    \label{eq:OverFactorization}
    \widetilde{T}_d^j S_\Ab = 2^{-j} \, S_{\Bb_j} \widetilde{T}_d^j,
  \end{equation}
  with the augmented Taylor operator $\widetilde{T}_d^j$ from Definition \ref{def:augmented}. Furthermore $\dim\,\cE_{\Bb_j}=1,j=d\ldots,n,$ and
  the factorization \eqref{eq:OverFactorization} is independent of the concrete spectral polynomials in
  \eqref{eq:SpecCond}.
\end{theorem}

\begin{pf}
Denote by $p_k,k=0,\ldots,n$, the spectral polynomials from
\eqref{eq:SpecCond}. Due to their normalization we have $p^{(k)}_k=1$.

We prove this result by induction on $j$.
From Remark \ref{rem:Taylor} we have $\widetilde{T}_d^d=\widetilde{T}_d$ and
the existence of $S_{\Bb_d}$ follows from \cite{merrienSauer12:_from_hermit}, see \eqref{eq:Taylor_factorization}. Also $\dim\,\cE_{\Bb_d}=1$ follows from \cite{merrienSauer12:_from_hermit}.
This 
shows the case $j=d$.

We assume that the theorem is true for $j$ and prove it for $j+1$.
Lemma \ref{lem:Taylor_applied} and Corollary \ref{cor:augmented_eigenvectors}
imply
\begin{equation*}
 \widetilde{T}_d^j\vb(p_{j+1})=\eb_d + \yb_{j+1-d}.
\end{equation*}
The spectral condition implies
  \begin{align*}
   2^{-j-1}(\eb_d + \yb_{j+1-d})&= 2^{-j-1}\widetilde{T}_d^j \vb(p_{j+1})
   =\widetilde{T}_d^j S_\Ab \vb(p_{j+1}) = 
    2^{-j} \, S_{\Bb_j} \widetilde{T}_d\vb(p_{j+1})\\
    &=2^{-j} \, S_{\Bb_j} (\eb_d + \yb_{j+1-d}),
  \end{align*}
and thus
 \begin{equation*}
  2\, S_{\Bb_{j}}\left( \eb_d + \yb_{j+1-d}\right) = \eb_d + \yb_{j+1-d}.
 \end{equation*}
Therefore $\eb_d + \yb_{j+1-d}$ lies in $\cE_{2\,\Bb_j}$ and since by assumption
the dimension of this space is $1$, it is spanned by $\eb_d + \yb_{j+1-d}$. Now we use Lemma \ref{lem:rank1_factorization} to factorize further. The Gau{\ss} matrix
\begin{equation}\label{eq:Gauss}
\Ib_d + \yb_{j+1-d} \eb_d^T =
\begin{pmatrix}
  1 & & & \yb_{j+1-d,0} \\
  & \ddots & & \vdots \\
  & & 1 & \yb_{j+1-d,d-1} \\
  & & & 1 \\
\end{pmatrix},
\end{equation}
is an $\cE_{2\, \Bb_j}$-generator.
It is easy to check that $\left( \Ib_d + \yb_{j+1-d} \eb_d^T
\right)^{-1} = \Ib_d - \yb_{j+1-d} \eb_d^T$.
Lemma \ref{lem:rank1_factorization} thus implies that there exists a subdivision
operator $S_{\Bb_{j+1}}$ such that
\begin{equation}\label{eq:vector_factor}
 2 D\left(\Ib_d - \yb_{j+1-d} \eb_d^T\right) S_{\Bb_j}= 
 S_{\Bb_{j+1}}D\left(\Ib_d - \yb_{j+1-d} \eb_d^T\right)
\end{equation}
and such that $\dim\,\cE_{\Bb_{j+1}}=1$.
The factorization \eqref{eq:vector_factor} further implies
\begin{align*}
 D\left(\Ib_d - \yb_{j+1-d} \eb_d^T\right)\widetilde{T}_d^j S_{\Ab}&=
 2^{-j}D\left(\Ib_d - \yb_{j+1-d} \eb_d^T\right)S_{\Bb_j}\widetilde{T}_d^j\\
 & =
 2^{-j-1}S_{\Bb_{j+1}}D\left(\Ib_d - \yb_{j+1-d} \eb_d^T\right)\widetilde{T}_d^j.
\end{align*}
From Lemma \ref{lem:augmented_iterated} we know that 
$D\left(\Ib_d - \yb_{j+1-d} \eb_d^T\right)\widetilde{T}_d^j=\widetilde{T}^{j+1}_d$.
This concludes the induction.
\end{pf}
\begin{remark}\label{rem:greg_operator}
  Theorem \ref{thm:factorization} for $d = 1$ and Definition \ref{def:augmented} give
  \begin{equation*}
      \widetilde{T}_1^j = \begin{pmatrix}
        \Delta & - \sum_{k=0}^{n-1}G_k^1 \Delta^k \\
        0 & \Delta^j
      \end{pmatrix} = \left(\begin{array}{cc} 0 & 1 \\ 1 & 0\end{array}\right) \mathcal{G}^{[j]},
  \end{equation*}
  where $G^1_k$ are the Gregory coefficients, see Section \ref{sec:Stirling}, and $\mathcal{G}^{[j]}$ is the \emph{Gregory operator} derived in \cite{moosmueller18b}. Therefore, Theorem \ref{thm:factorization} generalizes \cite{moosmueller18b}. Note that the matrix $\left(\begin{array}{cc} 0 & 1 \\ 1 & 0\end{array}\right)$ appears since we use \eqref{eq:Gauss} to transform to $\eb_1$ while \cite{moosmueller18b} uses an equivalent factorization as in Lemma \ref{lem:rank1_factorization} where a transform to $\eb_0$ is needed. The factorization is correct in both cases.
\end{remark}
\begin{remark}
 The paper \cite{jeong19} proves factorization and convergence results for level-dependent Hermite subdivision schemes of dimension $d=1$. In particular it considers schemes \eqref{eq:sds}, where the operators $S_{\Ab^{[j]}}, j \in \NN$, are
 \emph{not} restricted to the form \eqref{eq:stationary_mask}. From results 5.6 -- 5.8 in \cite{jeong19} we can deduce an interesting connection to the augmented Taylor operator.
 
 Consider a subdivision operator $S_{\Ab^{[j]}}$ of dimension $d=1$ which reproduces $\{1,x,e^{\lambda x}\}$ (this implies that it satisfies the spectral condition \eqref{eq:SpecCond} with the functions $1,x$ and $e^{\lambda x}$). Then there exists a subdivision operator $S_{\Bb^{[j]}}$ such that
 \begin{equation*}
     R^{[j+1]}S_{\Ab^{[j]}}=2^{-2}\zeta(j)S_{\Bb^{[j]}}R^{[j]},
 \end{equation*}
 where $R^{[j]}$ is given by
 \begin{equation*}
     R^{[j]}= \left(\begin{array}{cc} 0 & \delta_j \Delta \\ \Delta & -1-\eta(j)\Delta \end{array} \right),
 \end{equation*}
 with $\zeta,\eta$ from \cite[Proposition 5.8 (ii)]{jeong19}:
 \begin{align*}
     \zeta(j) = \frac{2}{e^{\lambda 2^{-j-1}}+1}, \quad
     \eta(j) = \frac{e^{\lambda 2^{-j}}-1-\lambda 2^{-j}}{\lambda 2^{-j}(e^{\lambda 2^{-j}}-1)}
 \end{align*}
 and
 \begin{equation*}
          (\delta_j \cb)(\alpha) = e^{-\lambda 2^{-j}}\cb(\alpha +1)-\cb(\alpha), \quad \cb \in \ell^{2}(\ZZ).
 \end{equation*}
 Furthermore, with Definition \ref{def:augmented}, \eqref{eq:coeff_int1} and \eqref{eq:coeff_0}, we obtain
 \begin{equation}\label{eq:level_dep_limit}
     \lim_{j \to \infty}R^{[j]} = \left(\begin{array}{cc} 0 & \Delta^2 \\ \Delta & -1-2^{-1}\Delta \end{array}\right)=\mathcal{G}^{[2]}=
     \left(\begin{array}{cc} 0 & 1 \\ 1 & 0\end{array}\right)
     \widetilde{T}_1^2.
 \end{equation}
 The transformation $\left(\begin{array}{cc} 0 & 1 \\ 1 & 0\end{array}\right)$ and the Gregory operator $\mathcal{G}^{[2]}$ (cf. \cite{moosmueller18b}) appear for the same reason as in Remark \ref{rem:greg_operator}.
 
Eq.\ \eqref{eq:level_dep_limit} implies that factorizing level-dependent schemes of dimension $d=1$ reproducing $\{1,x,e^{\lambda x}\}$ is connected to factorizing stationary schemes of the same dimension reproducing $\{1,x,x^2\}$ via limits. 
The level-dependent factorizations of \cite{jeong19} thus depend on $S_{\Ab}$ satisfying a type of overreproduction, in contrary to the factorizations of \cite{cotronei18}.

Through this overreproduction, the connection to the augmented Taylor operator is not surprising, considering that the cancellation operator for level-dependent Hermite schemes reproducing exponentials of \cite{conti16} converges to the Taylor operator, cf.\ \cite[Corollary 2]{conti16}. This also indicates that a generalization of \cite{jeong19} to $d>1$ and multiple exponentials, has to be an operator which converges to $\widetilde{T}^j_d$. 
\end{remark}

A generalization of the spectral condition \eqref{eq:SpecCond} to so-called \emph{spectral chains} is proposed in \cite{MerrienSauer18S}. We mention two special spectral chain for which the augmented Taylor operator can be computed easily. Consider a subdivision operator $S_{\Ab}$ with spectral chain
\begin{equation}\label{eq:new_spec_poly}
\vb(p_k) =
\begin{pmatrix}
  \Delta^j p_k : j=0,\dots,d
\end{pmatrix}^T
, \quad k=0,\ldots n.
\end{equation}
This implies that $S_{\Ab}$ satisfies \eqref{eq:SpecCond} with \eqref{eq:new_spec_poly}. In this case $S_{\Ab}$ factorizes with respect to a \emph{complete Taylor operator} of the form
$$
\begin{pmatrix}
  \Delta & -1 \\
  & \ddots & \ddots \\
  & & \Delta & -1 \\
  & & & \Delta
\end{pmatrix},
$$
cf.\ \cite{MerrienSauer18S}.
Applying the augmented Taylor construction, analogous to Theorem \ref{thm:factorization}, we obtain that $S_{\Ab}$ factorizes with respect to the operators
$$
\begin{pmatrix}
  \Delta & -1 & &\\
  & \ddots & \ddots &\\
  & & \Delta & -1 \\
  & & & \Delta^{j+1-d}
\end{pmatrix}, \quad j=d,\ldots,n.
$$
Note that in this case all vectors $\yb$ are zero.

We also consider the following spectral chain which is connected to B-Splines:
\begin{equation}\label{eq:new_spec_poly2}
\vb(p_k) =
\begin{pmatrix}
  \Delta^j p_k(\cdot - j) : j=0,\dots,d
\end{pmatrix}^T
, \quad k=0,\ldots n,
\end{equation}
see \cite{MerrienSauer18S}. In \cite{MerrienSauer18S} it is proved that a subdivision operator $S_{\Ab}$ with spectral chain \eqref{eq:new_spec_poly2} factorizes with respect to the generalized Taylor operator
$$
\begin{pmatrix}
  \Delta & -1  &  \cdots & -1\\
  & \ddots & \ddots &  \vdots \\
  & & \Delta & -1 \\
  & & & \Delta
\end{pmatrix}.
$$
With the augmented Taylor construction we obtain that $S_{\Ab}$ factorizes with respect to
$$
\begin{pmatrix}
  \Delta & -1  &  \cdots & -1 & -1-\Delta\\
  & \ddots & \ddots &  \vdots &\vdots \\
  & &  \Delta &  -1 &-1 -\Delta \\
  & & & \Delta &-1 -\Delta \\
  & & & & \Delta^{j+1-d}
\end{pmatrix}, \quad j=d+1,\ldots,n.
$$
Note that in this case $\yb_0=\left(1.\ldots,1,0\right)^T$ and $\yb_j=\mathbf{0}, j>0$.





\section{Interpretation of the augmented Taylor operator}
The coefficients $G_n^k$ appear in the following approximations for integrating functions $f$ (see \cite{phillips72,salzer47}):
\begin{align}\label{integral_coeff}
    \int_{x_0}^{x_1}\int_{x_0}^{x_2}\cdots\int^{x_{k}}_{x_0}f(x)dxdx_{k}\cdots dx_2
    =(x_1-x_0)^k\sum_{n=0}^m G_n^k\Delta^{n}f(x_0) +R_m^kf(x_1;x_0),
\end{align}
where $R_m^kf(x_0;x_1)$ denotes the remainder term.
Via \eqref{integral_coeff} we derive an interpretation of the augmented Taylor operator $\widetilde T _d^n$ (Theorem \ref{thm:interpret_augm_taylor}).

Let $f\in C^d(\RR)$ and denote by $\cT_n f(x_1;x_0)$ its $n$-th Taylor polynomial, i.e.\ 
\begin{equation*}
    \cT_n f(x_1;x_0)=\sum_{k=0}^n\frac{f^{(k)}(x_0)}{k!}(x_1-x_0)^k, \quad n=0,\ldots, d.
\end{equation*}
In analogy we define
\begin{equation}\label{eq:I}
    \cI^k_nf(x_1;x_0):=\sum_{m=0}^nG_m^k\Delta^{m}f(x_0)(x_1-x_0)^k.
\end{equation}
Thus \eqref{integral_coeff} becomes
\begin{equation*}
     \int_{x_0}^{x_1}\int_{x_0}^{x_2} \cdots \int_{x_0}^{x_k} f(x)dxdx_k\cdots dx_2=\cI^k_nf(x_1;x_0)+R_n^kf(x_1,x_0).
\end{equation*}
It is easy to see that
\begin{equation*}
    \int_{x_0}^{x_1}\int_{x_0}^{x_2} \cdots \int_{x_0}^{x_{d-j}} f^{(d)}(x)dxdx_{d-j}\cdots dx_2
    =f^{(j)}(x_1)-\cT_{d-j-1}f^{(j)}(x_1;x_0),  
\end{equation*}
for $j=0,\ldots,d-1.$
Thus we get
\begin{equation}\label{eq:IT}
    \cI_n^kf^{(d)}(x_1;x_0)=f^{(d-k)}(x_1)-\cT_{k-1}f^{(d-k)}(x_1;x_0)-R_n^kf^{(d)}(x_1,x_0).
\end{equation}
From \cite{merrienSauer12:_from_hermit} we know
\begin{equation*}
    \left(\widetilde{T}_d\vb(f)(x)\right)_j=f^{(j)}(x + 1)-\cT_{d-j}f^{(j)}(x+1; x),  \quad j=0,\ldots,d-1,
\end{equation*}
i.e.\ the remainder term, when Taylor expanding $f^{(j)}(x+1)$ at $x$ with order $d-j$.
Now consider the augmented Taylor operator in view of \eqref{eq:I} and \eqref{eq:IT}:
\begin{align*}
    \left(\widetilde{T}_d^n\vb(f)(x)\right)_j&=
    f^{(j)}(x+1)-\cT_{d-j-1}f^{(j)}(x+1;x)-\sum_{k=0}^{n-d}G_k^{d-j}\Delta^kf^{(d)}\\
    &=f^{(j)}(x+1)-\cT_{d-j-1}f^{(j)}(x+1;x)-\cI^{d-j}_{n-d}f^{(d)}(x+1;x)\\
    &=R^{d-j}_{n-d}f^{(d)}(x+1;x),
\end{align*}
that is, the remainder term, when integrating $f^{(d)}, (d-j)$-times with precision $n-d$.
We summarize this result in the following theorem.
\begin{theorem}\label{thm:interpret_augm_taylor}
Let $f \in C^d(\RR)$. Then
\begin{equation*}
    \widetilde{T}_d^n \vb(f)(x) =
    \widetilde{T}_d^n
    \begin{pmatrix}
      f(x) \\
      f'(x) \\
      \vdots \\
      f^{(d)}(x)
    \end{pmatrix}    
    =
    \begin{pmatrix}
      R^{d}_{n-d}f^{(d)}(x+1;x) \\[1em]
      R^{d-1}_{n-d}f^{(d)}(x+1;x) \\[1em]
      \vdots \\[1em]
      R^{0}_{n-d}f^{(d)}(x+1;x)
    \end{pmatrix},
    \quad x\in \RR,
\end{equation*}
with the remainder terms $R_d^{d-j}f, j=0,\ldots,d$, given in \eqref{integral_coeff}.
\end{theorem}

\bibliographystyle{plainnat}


\providecommand{\bysame}{\leavevmode\hbox to3em{\hrulefill}\thinspace}
\providecommand{\MR}{\relax\ifhmode\unskip\space\fi MR }
\providecommand{\MRhref}[2]{%
  \href{http://www.ams.org/mathscinet-getitem?mr=#1}{#2}
}
\providecommand{\href}[2]{#2}

\end{document}